\documentclass[pdflatex,sn-mathphys-num]{sn-jnl}% Math and Physical Sciences Numbered Reference Style

\usepackage[T1]{fontenc}
\usepackage[utf8]{inputenc}

\usepackage{graphicx}%
\usepackage{multirow}%
\usepackage{amsmath,amssymb,amsfonts}%
\usepackage{amsthm}%
\usepackage{mathrsfs}%
\usepackage[title]{appendix}%
\usepackage{xcolor}%
\usepackage{textcomp}%
\usepackage{manyfoot}%
\usepackage{booktabs}%
\usepackage{algorithm}%
\usepackage{algorithmicx}%
\usepackage{algpseudocode}%
\usepackage{listings}%

\usepackage{tikz}
\usetikzlibrary{math}  % To define variable quantities
\usetikzlibrary{arrows.meta}  % nice arrowheads
\usepackage{pgfplots}
\usetikzlibrary{plotmarks}

\definecolor{purpleG}{rgb}{0.406, 0, 0.566}

\usepackage{dsfont}   % ds font for \one
\usepackage{mathtools} % for \coloneqq :=
\usepackage{enumitem}% for typographic control of \begin{enumerate}
	\usepackage{hyperref}% for links
	\usepackage{subcaption}% for subfigure

	\theoremstyle{plain}
	\newtheorem{theorem}{Theorem}
	\newtheorem*{theorem*}{Theorem}
	\newtheorem{proposition}[theorem]{Proposition}
	\newtheorem{lemma}[theorem]{Lemma}
	\newtheorem{corollary}[theorem]{Corollary}
	\theoremstyle{definition}
	\newtheorem{definition}{Definition}

	\theoremstyle{remark}
	\newtheorem{remark}{Remark}
	
	\newcommand{\one}{\mathds{1}}        % vector of all 1
	\DeclareMathOperator{\inv}{inv}

	\newcommand{\errT}{{\mathrm{err}}_{T}}  % backward error
	\newcommand{\erry}{{\mathrm{err}}_{y}}   % forward error
	
	\newcommand{\din}{d_{\mathrm{in}}}          % vector of the in-degrees
	\newcommand{\dout}{d_{\mathrm{out}}}     % vector of the out-degrees
	
	\newcommand{\dmin}{d_{\min}}      % minimum degree
	\newcommand{\dmax}{d_{\max}}     % maximum degree

\begin{document}
		
		\title{Convergence and acceleration of a nonlinear fixed-point iteration for computing the Fitness Centrality of general graphs}
	
	\author*[1]{\fnm{Nikita} \sur{Deniskin}}\email{nikita.deniskin@sns.it}
	
	\author[1]{\fnm{Michele} \sur{Benzi}}\email{michele.benzi@sns.it}
	
	\affil[1]{\orgdiv{Faculty of Sciences}, \orgname{Scuola Normale Superiore}, \orgaddress{\street{Piazza dei Cavalieri, 7}, \city{Pisa}, \postcode{56126}, \country{Italy}}}

	\abstract{We establish the global convergence of the (non-homogeneous) Fitness Centrality algorithm for general graphs, deriving an
		explicit convergence bound for the corresponding fixed-point iteration. 
		From the computational point of view, we show how convergence can be dramatically
		improved by Anderson acceleration and by switching to Newton's method once a sufficiently good approximation to the fixed point has been found.
		The efficacy of this strategy is illustrated by numerical experiments on different types of graphs.}
	
	
	
	% 	05C50   	Graphs and linear algebra (matrices, eigenvalues, etc.)
	% 65B05   	Acceleration of convergence in numerical analysis - Extrapolation to the limit, deferred corrections 
	
	\maketitle
	
	\section{Introduction}
	One of the most fundamental problems in Network Science is that of {\it node centrality}, i.e., identifying the most ``important" nodes
	in a given network according to some definition of importance \cite{Newman2018}.  Common criteria to consider a node important include
	playing a crucial role in connecting many other nodes with each other, or having a high probability of being visited often in the course 
	of a random walk on the network. Widely used centrality measures of this type are betweenness centrality \cite{Freeman1977}, eigenvector centrality
	\cite{Bonacich1987}, PageRank \cite{BrinPage1998}, Katz centrality \cite{Katz1953}, and subgraph centrality \cite{EstradaRV2005}, to name a few.

	All these measures have found
	extensive use in a number of applications in the natural and social sciences, as well as in engineering. Computing these centralities requires
	finding shortest paths between pairs of nodes in the case of betweenness centrality, or performing linear algebra operations involving the solution
	of linear systems, eigenvalue computations, or evaluation of matrix functions in other cases \cite{BenziBoito2020}.  For large graphs this can
	be computationally expensive, and much effort has been spent in developing efficient algorithms. \\
	
	Centrality measures have also been developed with an eye to identifying weak spots, or vulnerabilities, in complex networks. One such example
	is {\it Fitness Centrality} (FC), first proposed in \cite{FCNetworks2025}; see also the recent PhD thesis \cite{Calo_PhD2026}.
	This measure is rooted in earlier work in the field of economics; see
	\cite{def1EFC2012,def2EFC2013}, where the notion of {\it Economic Fitness Complexity} (EFC) was originally proposed.  In these works the focus is on
	the bipartite graph whose nodes represent countries and products. Two coupled nonlinear maps are introduced, whose fixed points measure countries'
	fitness and products' complexity, economic notions which we do not elaborate on here. Coupled iterations are used to compute the two fixed points,
	and numerical experiments show convergence regardless of the initial values, suggesting existence and uniqueness, but no formal proofs are given
	in \cite{def1EFC2012,def2EFC2013}.\\
	
	In \cite{ConvFCAlgorithm2016}, the authors consider the use of the EFC algorithm for a generic bipartite graph and observe that the
	structure of the adjacency matrix must satisfy some restrictions in order for the iterates to define non-zero limiting values. 
	A further step was taken in \cite{regularizedEFC2018}, where a regularization parameter $\delta > 0$ was introduced so as to improve the robustness 
	and stability of the FC algorithm. The resulting method is called {\it  Non-Homogeneous Economic Fitness Complexity}.
	In \cite{regularizedEFC2018}, the authors observed that the fixed points of the algorithm depend mildly on $\delta$ and are virtually independent 
	of it as soon as $\delta$ becomes sufficiently small.  They also showed that the spectral radius of the Jacobian of the nonlinear
	map governing the coupled iteration computed at the fixed points is less than 1, but gave no proof of  global convergence. \\
	
	The close relationship existing between the Fitness Complexity algorithm and the classical Sinkhorn-Knopp iteration was uncovered in \cite{EquivFC_SK2024}.
	The goal of the Sinkhorn-Knopp iteration is to find two positive diagonal matrices, $D_1$ and $D_2$, such that the scaled matrix $D_1AD_2$ is doubly stochastic, where
	$A$ is a given nonnegative matrix \cite{SK1967}. It is a widely used algorithm with a number of important practical applications \cite{Knight2008}. The essential
	equivalence of the two algorithms helps explain several properties of the Fitness Complexity iteration, including the fact that its convergence depends on
	the structure of the underlying matrix. We note, however, that the regularization technique adopted in \cite{regularizedEFC2018} was not considered in
	\cite{EquivFC_SK2024}, and that the convergence theory for Sinkhorn-Knopp does not apply to Non-Homogeneous Economic Fitness Complexity.\\
	
	In all of the papers on Fitness Complexity and its variants mentioned thus far, the discussion is centered around economic applications and limited to
	bipartite networks. In \cite{FCNetworks2025}, the authors introduce an extension of the algorithm to general, non-bipartite  networks, not necessarily from 
	economics.\footnote{The algorithm in  \cite{FCNetworks2025} also extends the so-called {\it Economic Complexity Index} (ECI), which is used for
		network clustering. For our purposes, ECI need not be discussed further.}
	This extension leads to the notion of {\it Fitness Centrality}, which the authors describe as a novel centrality measure that can be used
	for assessing node vulnerability. Indeed, extensive numerical experiments show that, on one hand, Fitness Centrality (FC) strongly
	correlates with node degree, while on the other it tends to assign higher values to nodes that are connected to many nodes of low degree.  Since
	removing such nodes from a network results in a highly disconnected network, the authors of  \cite{FCNetworks2025} propose the use of Fitness
	Centrality as a measure of vulnerability. Again, regularization of the iteration (via a parameter $\delta > 0$) is found necessary to ensure robustness of the iterative 
	process and stability of the fixed point. If $A= [A_{ij}]$ denotes the $n\times n$ adjacency matrix of the graph (possibly weighted, directed or undirected), the 
	fixed-point iteration for computing the Fitness Centrality of nodes takes the (componentwise) form
	\begin{equation}\label{eq:NHFC}
		x_i^{(k+1)} = \delta + \sum\limits_{j=1}^n A_{ij} \frac{1}{x_j^{(k)}}, \quad k = 0, 1, \ldots , \quad 1\le i \le n\,,
	\end{equation}
	where $x^{(0)}$ is an initial vector with positive components. \\
	
	In the recent paper \cite{CostFunctions2026}, the authors introduce a framework based on potential functions that allows one to interpret the
	fixed point of the non-homogeneous (i.e., regularized) FC algorithm as the minimizer of a suitable energy function in the case of an undirected graph.
	The authors prove the strict convexity of the energy, from which they 
	obtain existence and uniqueness of the minimum.  The convergence of the nonlinear fixed-point iteration (in its regularized variant) is discussed in
	some special cases, but no global convergence results are given. As the authors write,
	
	\begin{quote}
		While a general proof of convergence is difficult for arbitrary network structures, this property can be investigated empirically by analyzing 
		the behavior of multiple trajectories initialized from different starting points .... The results indicate that, regardless of the initial condition in the phase space, 
		all trajectories converge to the same point. Although a formal proof is lacking, this evidence suggests that global convergence may hold more generally. 
		At the same time, we cannot rule out the existence of nonstationary orbits (e.g., periodic ones) that attract the trajectories ... for suitable choices of the 
		initial configuration.
	\end{quote} 
	
	As an alternative, a steepest descent algorithm for energy minimization is proposed. The robustness and efficiency of this approach, however, remain to be 
	investigated.\\
	
	In this paper we settle the question of convergence for the regularized FC algorithm given by Eq. \eqref{eq:NHFC} for
	arbitrary graphs, including weighted and directed ones.
	More precisely, we prove linear convergence of \eqref{eq:NHFC} to the unique fixed point  of the map
	\[T (x) = \delta \one + A \cdot  \inv(x),\]
	where $\one$ is the vector of all ones and $\inv(x)$ is the vector whose components are the reciprocals of the components of $x$ (assumed nonzero), for any initial vector $x^{(0)}$ with positive components.
	We also provide an explicit bound for the rate of convergence. 
	
	While our theoretical result guarantees a linear convergence of the fixed-point iteration, in practice the convergence rate of the FC algorithm can be exceedingly slow. To address this, we study several acceleration techniques. One class is that of \emph{linear acceleration} techniques, which construct the new iterate by means of a linear combination of the previous iterates. We apply Anderson acceleration and we construct a simple \emph{ad hoc} scheme tailored to the FC iteration, which we call \emph{averaged strategy}. Other techniques rely on second-order methods, which typically achieve superlinear convergence but have a higher per-iteration cost and are not guaranteed to have global convergence. Among these, we focus specifically on Newton's method.
	We demonstrate experimentally that these methods achieve rapid convergence. In particular, good performance can be obtained by a judicious combination of Anderson acceleration and Newton's method.
	
	The remainder of the paper is organized as follows. In Section \ref{sec:prelim} we provide basic definitions and notation used throughout the paper.
	In Section \ref{sec:theory} we carry out  a detailed analysis of the convergence of the fixed-point iteration, while in Section \ref{sec:comp} we discuss techniques to accelerate the convergence, including their computational aspects. Numerical experiments on a variety of graphs are presented in Section \ref{sec:num_exp},
	and concluding remarks are given in Section \ref{sec:conclusions}.
	
	\section{Preliminaries} \label{sec:prelim}
	We denote by $\mathbb{R}_+$ the set of all positive real numbers. The \emph{positive orthant} of the $n$-dimensional space, i.e., the set of vectors with all positive components, is denoted by~$\mathbb{R}_+^n$. 
	
	\subsection{Graph notions}
	We start by introducing background definitions from graph theory. 
	\begin{definition}
		A \emph{graph}  $G=(V,E,\omega)$ is a triple, where $V$ is the set of \emph{vertices} (or \emph{nodes}), $E\subseteq V \times V$ denotes the \emph{edges}, and $\omega: E \to \mathbb{R}_+$ is the function that assigns weights to the edges. 
	\end{definition}
	We consider a general case where graphs are \emph{directed}, i.e., edges $(i,j)$ and $(j,i)$ may be present or not independently, and \emph{weighted} with positive weights. The weight of an edge indicates how strong the connection is between its endpoints: the larger the weight, the easier it is to transmit information from one vertex to the other.
	\begin{definition}
		Let $G=(V,E,\omega)$ be a graph with $|V|=n$ vertices. 
		The \emph{adjacency matrix} of $G$ is the matrix $A\in\mathbb{R}^{n\times n}$ defined by  
		\[A_{ij} = \begin{cases}
			\omega(i,j) & \text{ if } (i,j) \in E,\\
			0 & \text{ if } (i,j) \notin E.
		\end{cases}\]
	\end{definition}
	 The graph may contain self-loops; however, multiple edges between a fixed pair of vertices are not allowed.	 
	Throughout this paper, $n$ will denote the number of vertices of the graph, and $\one\in\mathbb{R}^n$ will denote the vector of all ones. The graph is \emph{unweighted} if $\omega(i,j) = 1$ for all $(i,j)\in E$. For simplicity, the reader may think of $G$ as being undirected, unweighted, and without loops, if desired.  
	
	\begin{definition}
		The \emph{out-degree} of a vertex $i$ is the sum of the weights of the \emph{outgoing} edges from $i$, i.e., the edges of the form $(i,j)$ for some $j$. The \emph{in-degree} of a vertex $i$ is  the sum of the weights of the \emph{incoming} edges to $i$, i.e., the edges of the form $(j,i)$ for some $j$.
		The out-degree vector is $\dout = A \one$, and the in-degree vector is $\din = A^\top \one$. 
		The minimum and maximum out-degrees are denoted, respectively, by 
		\[\dmin = \min\limits_i\;\dout(i), \quad\quad \dmax = \max\limits_i\;\dout(i).\]
	\end{definition}
	For an unweighted graph, the out-degree is just the number of outgoing edges, and similarly the in-degree is the number of incoming edges. 
	
	In an undirected graph we have $A=A^\top$, which implies that $\dout = \din$.

	\subsection{Fitness Centrality iteration}	
	Given a vector $x\in \mathbb{R}^n$ with no zero entries, we use the shorthand $\inv(x)$ to denote the vector with $i$-th component $[\inv(x)]_i  = 1/x_i$. 
	
	First, we recall the definition of the iterative map, which is essential for defining Fitness Centrality. 
	
	\begin{definition}\label{def:FC iteration}Given a graph with adjacency matrix $A$ and a parameter $\delta \in \mathbb{R}_+$,  the \emph{iteration map} of Fitness Centrality is $T:\mathbb{R}_+^n \to \mathbb{R}_+^n$, defined component-wise as
		\begin{equation}\label{eq:FC iteration}
			[T(x)]_i \coloneqq \delta + \sum\limits_{j=1}^n A_{ij} \frac{1}{x_j}.
		\end{equation}
		In vector form the expression is $T(x) \coloneqq \delta \one + A \cdot  \inv(x)$. 
	\end{definition}
	
	The map $T$ is well defined; that is, if $x\in\mathbb{R}_+^n$, then all components of $T(x)$ are also positive. Note that in the original notion of Economic Fitness Complexity \cite{def1EFC2012, def2EFC2013}, the parameter $\delta$ was absent; the iterative map of the Economic Fitness Complexity algorithm is equivalent to setting $\delta=0$ in \eqref{eq:FC iteration}.
	
	\begin{definition}\label{def:FC_step}
		For a starting vector $x^{(0)} \in \mathbb{R}_+^n$, the \emph{Fitness Centrality iterations} are defined recursively as $x^{(k)} = T(x^{(k-1)})$ for $k\geq 1$. We will call a single step of this iteration a \emph{standard step}.
	\end{definition}
	If the iterations $x^{(k)}$ have a limit as $k$ grows, then the limit must be a fixed point of the map $T$. A priori, it is not obvious whether the limit exists, or whether $T$ has one or multiple fixed points. 
	\begin{definition}\label{def:EFCentrality}
		The \emph{Fitness Centrality} of a graph $G$ (with parameter $\delta \in \mathbb{R}_+$) is a fixed point of the iteration determined by $T$. In other words, it is a vector $y\in\mathbb{R}_+^n$ satisfying 
		\[T(y) = y.\]
	\end{definition}
	
	Until recently, the existence and uniqueness of the solution were not guaranteed. In \cite{CostFunctions2026}, it was proved that for an undirected (possibly weighted) graph, there exists a unique fixed point. The proof is based on the construction of a convex potential, whose stationary point is the unique fixed point of $T$.
	However, whether the iteration $x^{(k)} = T(x^{(k-1)})$ converges to $y$ from an arbitrary starting point $x^{(0)} \in \mathbb{R}^n_+$ remained an open question. 
	In the next section, we prove that the iterates converge globally for any graph, including directed and weighted graphs, and derive explicit convergence bounds. We then propose computational strategies that lead to fast convergence to the unique fixed point.

	\section{Theoretical aspects of the convergence} \label{sec:theory}

	In this section, we analyze the convergence of the iteration from a theoretical perspective. We begin by establishing lower and upper bounds for each iterate of $T$, which allows us to construct a sequence of nested sets $E_k$ that contain the $k$-th iterate $x^{(k)}$. We then study the asymptotic behavior of $E_k$ as $k\to \infty$. Using this framework, we prove that starting from any initial condition, the iterates converge to the unique fixed point, and we derive an explicit upper bound on the convergence rate. 
	
	\subsection{Bounds of the iteration}
	With the definitions of the minimum and maximum out-degrees, any row sum of $A$ can be bounded as $\dmin \leq  [A \one]_i \leq \dmax$. This allows us to estimate $T(x)$ when bounds on $x$ are known. 
	
	\begin{proposition}\label{prop:1-step iteration bound}
		Consider $x\in \mathbb{R}_+^n$, and suppose that for $a,b> 0$ we have 
		\begin{equation}\label{eq:x is in [a,b]}
			a \leq x_i \leq b \quad \forall i.
		\end{equation}
		Then we have 
		\[ \delta + \frac{\dmin}{b} \leq [T(x)]_i \leq \delta + \frac{\dmax}{a} \quad \forall i.\]	
	\end{proposition}
	\begin{proof}
		From the hypothesis, equation \eqref{eq:x is in [a,b]}, we get 
		\[\frac{1}{b} \leq \frac{1}{x_i} \leq  \frac{1}{a}.\]
		Substituting this into \eqref{eq:FC iteration} yields the upper bound
		\[ [T(x)]_i = \delta + \sum\limits_{j=1}^n A_{ij} \frac{1}{x_j} \leq \delta + \sum\limits_{j=1}^n A_{ij} \frac{1}{a} = \delta + \frac{1}{a} \, \dout(i) \leq \delta + \frac{\dmax}{a}. \]
		Similarly, we obtain the lower bound
		\[ [T(x)]_i = \delta + \sum\limits_{j=1}^n A_{ij} \frac{1}{x_j} \geq \delta + \sum\limits_{j=1}^n A_{ij} \frac{1}{b} = \delta + \frac{1}{b} \, \dout(i) \geq \delta + \frac{\dmin}{b}. \]
	\end{proof}	
	To characterize the evolution of the upper and lower bounds, we define the following functions.
	\begin{definition}
		For any arbitrary positive real numbers $a$ and $b$, let $l_1, r_1$ be defined as 
		\[l_1(b) = \delta + \frac{\dmin}{b}, \quad\quad r_1(a) =\delta + \frac{\dmax}{a}.\]
		Let $l_2 = l_1 \circ r_1$ and $r_2 = r_1 \circ l_1$, that is 
		\[l_2(a) = \delta + \frac{\dmin}{\delta + \frac{\dmax}{a}}, \quad\quad r_2(b) = \delta + \frac{\dmax}{\delta + \frac{\dmin}{b}}. \]
	\end{definition}
	Proposition~\ref{prop:1-step iteration bound}  can then be reformulated as follows for a single- or two-step iteration.
	\begin{corollary}\label{cor:1,2-step iterations}
		If  $a \leq x_i^{(k)} \leq b$ for all components $i$, then 
		\[l_1(b) \leq x_i^{(k+1)} \leq r_1(a)\quad  \forall i,\]
		\[l_2(a) \leq x_i^{(k+2)} \leq r_2(b)\quad \forall i.\]
	\end{corollary}
	
	\begin{definition}\label{def:sets E_k}
		Starting from the set $E_{0} = [a_0, b_0]^n \subseteq \mathbb{R}_+^n$, the lower and upper bounds are defined recursively for $k \in \mathbb{N}$ by
		\[a_k = l_1\left(b_{k-1}\right), \quad b_{k} = r_1\left(a_{k-1}\right).\]
		The domain $E_{k}$ is defined as
		\[E_{k} = [a_{k}, b_{k}]^n.\]
	\end{definition}
	Combining this definition with Proposition~\ref{prop:1-step iteration bound}, it follows that $T(E_k) \subseteq E_{k+1}$. By construction, if $x^{(0)}\in E_0$, then the $k$-th iterate $x^{(k)}$ lies in $E_k$. 

	We now analyze the evolution of the domains $E_k$ starting from any admissible initial vector $x^{(0)} \in \mathbb{R}^n_+$. Here, we need to use $\widetilde{E}_0 = \mathbb{R}^n_+ = (0,\infty)^n$, which is not covered by Definition~\ref{def:sets E_k} since $\mathbb{R}^n_+$ is not bounded. Note that the results of Proposition~\ref{prop:1-step iteration bound} and Corollary~\ref{cor:1,2-step iterations} also hold for $a=0$ and $b=\infty$, under the convention that $\frac{1}{0} = \infty$ and $\frac{1}{\infty}=0$. However, for the sake of rigor and to illustrate the process more clearly, we describe explicitly the first few iterations starting from $\widetilde{E}_0$. 
	
	After one step of the iteration, we obtain from \eqref{eq:FC iteration} that $x^{(1)}_i \geq \delta$  for all $i$ and no upper bound; that is, $\widetilde{E}_1 = (\delta,\infty)^n$. 
	After the second step, the lower bound for $x^{(1)}$ yields an upper bound for $x^{(2)}$, namely $x^{(2)}_i \leq \delta + \frac{\dmax}{\delta}$ for all $i$. Since the lower bound $x^{(2)}_i \geq \delta$ still holds, we have  $\widetilde{E}_2 = (\delta, \delta + \frac{\dmax}{\delta})^n$.  
	
	From this point onward, all subsequent sets are bounded, so we can use closed intervals instead of open ones and apply Proposition~\ref{prop:1-step iteration bound}. Indeed, we can assume without loss of generality that the starting vector lies in $E_0 = [\delta, \delta + \frac{\dmax}{\delta}]^n$, since after two steps of the iteration we would have $T^2(x^{(0)})\in E_0$ for any starting point $x^{(0)}\in \mathbb{R}^n_+$. 
	
	Having obtained a compact set $E_0$, we can prove the existence of a fixed point of $T$. 
	
	\begin{proposition}\label{prop:existence fixed point}
		The map $T$ admits a fixed point. 
	\end{proposition}
	\begin{proof}
		Consider any $x \in E_0 = [\delta, \delta + \frac{\dmax}{\delta}]^n$. Proposition~\ref{prop:1-step iteration bound} yields 
		\[T(x) \in E_1 = \left[\delta+ \frac{\dmin}{\delta + \frac{\dmax}{\delta}}, \delta + \frac{\dmax}{\delta}\right]^n.\]
		Note that $E_1 \subseteq E_0$, and so $T(E_0) \subseteq E_1 \subseteq E_0$. The set $E_0$ is compact, convex, and therefore homeomorphic to a closed ball in $\mathbb{R}^n$. 
		Thus, since $T$ is a continuous function from the set $E_0$ into itself, by Brouwer's fixed point theorem \cite[Theorem 6.3.2]{OrtRhe2000} there exists a fixed point $y$ for $T$; that is, $T(y) = y$.
	\end{proof}
	
	For now, we have only the existence of a fixed point; uniqueness will be shown in Corollary~\ref{cor:uniqueness fixed point}.
	
	\subsection{Limit of $E_k$ and estimate of $E_*$}
	In the above construction, after the first few iterations the size of the domains is reduced. The further iterations of $E_k$ also get progressively smaller, as the following proposition shows. 
	
	\begin{proposition}\label{prop:increasing sets E_k}
		Let $E_0 = [a_{0}, b_{0}]^n$ and construct $E_k=[a_{k}, b_{k}]^n$ for $k\geq 1$ as in Definition~\ref{def:sets E_k}. Suppose that $a_{0} \leq a_{1}$ and $b_{1} \leq b_{0}$.
		Then we have $E_{k+1} \subseteq E_{k}$ for every $k\geq 0$. If the inequalities are strict, i.e., $a_{0} < a_{1}$ and $b_{1} < b_{0}$, then the containment is also strict: $E_{k+1} \subsetneq E_{k}$.
	\end{proposition}
	\begin{proof}
		We prove it by induction on $k$. Compare the left and right boundaries of $E_{k+1}$ to those of $E_{k}$ for $k\geq 1$. 
		\[\begin{aligned}
			a_{k} = l_1(b_{k-1}) &= \delta + \frac{\dmin}{b_{k-1}}, \\
			a_{k+1} = l_1(b_{k}) &= \delta + \frac{\dmin}{b_{k}}.
		\end{aligned}\]
		By the induction hypothesis we have $b_{k}\leq b_{k-1}$, therefore  $a_{k+1} \geq a_{k}$. 
		The inequality $b_{k+1} \leq b_{k}$ can be proved in an analogous manner starting from  $a_{k} \geq a_{k-1}$. 
	\end{proof}
	
	Instead of using 1-step functions $l_1$ and $r_1$, it is more convenient to use the 2-step functions $l_2$ and $r_2$, as $a_{k+2} = l_2(a_{k})$ and $b_{k+2} = r_2(b_{k})$. Then, the asymptotic limit of $a_{k}$ and $b_{k}$ for $k\to \infty$ can be obtained by studying the fixed points of the maps $l_2$ and $r_2$, respectively, which we show to be unique and attractive. 
	\begin{proposition}\label{prop:limits of a_k and b_k}
		Starting from any $a_{0},b_{0} \in \mathbb{R}_+$, the iterations $a_{k}$ and $b_{k}$ converge, respectively, to
		\begin{equation}\label{eq:limits of a_k and b_k}
			\begin{aligned}
				&a_* \coloneqq \lim\limits_{k\to\infty} a_{k} =   \frac{1}{2\delta}\left( (\dmin-\dmax+\delta^2) + \sqrt{(\dmin-\dmax+\delta^2)^2+4\dmax\delta^2} \right), \\
				&b_* \coloneqq \lim\limits_{k\to\infty} b_{k} =  \frac{1}{2\delta}\left( (\dmax-\dmin+\delta^2) + \sqrt{(\dmax-\dmin+\delta^2)^2+4\dmin\delta^2} \right).
			\end{aligned}
		\end{equation}
	\end{proposition}
	\begin{proof}
		For the left bound, we have the iterative map
		\[l_2(z) = \delta +  \frac{\dmin z}{\delta z + \dmax}.\]
		The equation $l_2(z)=z$, after rearranging, becomes the quadratic equation
		\[z^2 \delta + z (\dmax-\dmin-\delta^2) - \dmax\delta = 0.\]
		There is  only one positive solution, which we call $a_*$:
		\[a_* = \frac{1}{2\delta}\left( (\dmin-\dmax+\delta^2) + \sqrt{(\dmin-\dmax+\delta^2)^2+4\dmax\delta^2} \right).\]
		The derivative of $l_2(z)$ is $l_2'(z) = \frac{\dmin\dmax}{(\delta z + \dmax)^2}$. Note that $0<l_2'(z) < 1$ for any $z >0$, so $a_*$ is an attractive point for the iterative sequence $\{a_{k}\}_{k\in\mathbb{N}}$. 
		
		For the right bound, the iterative map is
		\[r_2(z) = \delta +  \frac{\dmax z}{\delta z + \dmin}.\]
		The equation $r_2(z)=z$, after rearranging, becomes the quadratic equation
		\[z^2 \delta + z (\dmin-\dmax-\delta^2) - \dmin\delta = 0.\]
		It has only one positive solution, which we call $b_*$: 
		\[b_* = \frac{1}{2\delta}\left( (\dmax-\dmin+\delta^2) + \sqrt{(\dmax-\dmin+\delta^2)^2+4\dmin\delta^2} \right).\]
		
		The derivative of $r_2(z)$ is $r_2'(z) = \frac{\dmin\dmax}{(\delta z + \dmin)^2}$, which no longer satisfies $r_2'(z)<1$ for all $z \in \mathbb{R}_+$. Nevertheless, $r_2(z)$ is strictly increasing and concave on $z \in \mathbb{R}_+$, since $r_2'(z)>0$ and $r_2'(z)$ is strictly decreasing. 
		
		The threshold for the derivative we are interested in  is $w = \frac{\sqrt{\dmin\dmax}-\dmin}{\delta}$. Specifically, $r_2'(z)\geq 1$ for $z \in (0,w]$ and $r_2'(z) < 1$ for $z \in (w,+\infty)$.  It is easy to see that $b_*> w$, implying that $b_*$ is an attractive point; furthermore, if  $b_{k} \in (w,\infty)$ for some $k$, then the sequence converges to $b_*$. 
		
		It remains to examine the case where $b_{0} \in (0,w]$. Since $r_2'(z) \geq 1$ for $z \in [0,w]$, we have $r_2(z) \geq z + r_2(0) = z + \delta$. Therefore, if $b_{k}$ remains in $(0,w]$, then $b_{k+1} \geq b_{k} + \delta$. Thus, after a finite number of iterations the sequence must enter the interval $(w,\infty)$, where we know it will converge to $b_*$. 
	\end{proof}
	
	\begin{remark}\label{rem: the border maps swap a* and b*}
		Note that $l_2(l_1(b_*)) = l_1(r_1(l_1(b_*))) = l_1(r_2(b_*)) = l_1(b_*)$, so $l_1(b_*)$ is a fixed point for $l_2$. Since $a_*$ is the unique fixed point, it must hold that $l_1(b_*)=a_*$; similarly we have $r_1(a_*)=b_*$. That is, the maps $l_1$ and $r_1$ swap the asymptotic left and right boundaries $a_*$ and $b_*$.
	\end{remark}
	Since the expressions in \eqref{eq:limits of a_k and b_k} are quite cumbersome, we determine their leading-order behavior when $\delta$ is small and $\dmin<\dmax$. The first-order approximation of $a_*$ and $b_*$ is
	\begin{equation}\label{eq:a* and b* first order}
		a_* = \frac{\dmax\, \delta}{\dmax-\dmin} + O(\delta^3), \quad\;
		b_* = \frac{\dmax-\dmin}{\delta} + O(\delta), \quad\; \mathrm{for} \; \delta \to 0.
	\end{equation}
	In applications, the parameter $\delta$ is chosen to be small, but not too small in order to avoid slow convergence.  Under the assumption\footnote{The assumption is easily satisfied; for example, it is sufficient to have an edge with weight at least 1.} that $\dmin+\dmax\geq 1$, the approximation \eqref{eq:a* and b* first order} holds whenever $\frac{\delta(\dmax+\dmin)^{\frac{2}{3}}}{\dmax-\dmin} \ll 1$.
	
The case $\dmin=\dmax$ (i.e., a regular graph) produces a trivial result. Indeed, the Fitness Centrality scores of all vertices are the same, which gives us no insight into the graph.
	
	Proposition~\ref{prop:limits of a_k and b_k} can be reformulated focusing on the limit of the sets $E_k$ as $k\to \infty$, as shown in the following corollary.
	\begin{corollary}\label{cor:limit of x^k is in E*}
		For any initial condition $x^{(0)} \in \mathbb{R}^n_+$, if the generated sequence $\{x^{(k)}\}_{k\in \mathbb{N}}$ has a limit $x^* \coloneqq \lim\limits_{k\to\infty} x^{(k)}$, then $x^* \in E_{*} = [a_*,b_*]^n$. 
	\end{corollary}
	
	Unfortunately, we cannot obtain a tighter region than $E_*$. Even if we start in an extremely small rectangle with $b_{0}-a_{0} \ll 1$, the sequence of domains will still converge to $E_*$. In fact, in this case it is possible to prove a result opposite to Proposition~\ref{prop:increasing sets E_k}: if $E_0 \subseteq E_1$, then $E_{k} \subseteq E_{k+1}$ for $k\geq 0$. 
	
	This optimality is confirmed by the fact that there exist graphs that have a fixed point $y$ on the boundary of $E_*$.  Fix the parameter $\delta$ and the minimum and maximum degrees $\dmin,\dmax$; assume for simplicity that the degrees are integers. This determines the map $T$ and the limiting domain bounds $a_*, b_*$.
	Consider an undirected bipartite graph where $V = V_1 \cup V_2$, $|V_1| =  \dmax$ and $|V_2| = \dmin$. All possible edges between $V_1$ and $V_2$ are present, each with unit weight. Consequently, each vertex in $V_1$ has degree $\dmin$ and each vertex in $V_2$ has degree $\dmax$. 
	
	Consider $y \in \mathbb{R}^n$ where $y_i = a_*$ for vertices $i\in V_1$ and $y_j = b_*$ for vertices $j\in V_2$.  Then, for a vertex $i\in V_1$, we have
	\[ [T(y)]_i = \delta + \sum\limits_{j=1}^n A_{ij} \frac{1}{y_j} = \delta + \sum\limits_{j \in V_2} \frac{1}{y_j} = \delta + \frac{\dmin}{b_*} = a_* = y_i,\]
	where the second-to-last equality is due to Remark~\ref{rem: the border maps swap a* and b*}. Similarly we get $[T(y)]_j = b_* = y_j$ for $j \in V_2$. Therefore $y$ is a fixed point for $T$ with $y \in \partial  E_*$. 
	
	The same construction can be easily generalized to a graph with $|V_1| = k \dmax$ and $|V_2| = k\dmin$ for any $k \in \mathbb{Z}_+$, by using $k$ copies of the above example; if desired, some edge swaps can be performed to obtain a connected graph. 
	
	Fixing the values of $\delta$, $\dmin$, $\dmax$, and $n$ uniquely determines $E_*$. Multiple graphs share these minimum and maximum degrees but have, in general, different fixed points in $E_*$, with some lying on the boundary $\partial E_*$. Since $E_*$ is often a large domain in practice, it would be highly valuable to establish tighter bounds on the limits of the iterates (possibly under specific graph conditions) and, consequently, on the fixed points of $T$. This would provide deeper insight into the nature of Fitness Centrality for these graphs.

	\subsection{Proof of convergence}
	We are now ready to prove our main result. We show that $T^2$ is a contraction, which implies the uniqueness of the fixed point, as well as bounds on the convergence rate. 
	
	\begin{theorem}\label{thm:error bound iteration 2-step}
		Let $y$ be a fixed point of the map $T$. Let $E = [a,b]^n$ be such that  $T(E) \subseteq E$ and $y\in E$. Consider $x\in E$. 
		Let $\varepsilon$ be the maximum relative distance of the components of $x$ with respect to those of $y$, that is
		\[\varepsilon = \max\limits_{1\leq i\leq n} \, \frac{|x_i - y_i|}{y_i}.\]
		Then, the maximum relative distance of the components of $T^2(x)$ with respect to $y$ is
		\begin{equation}\label{eq:error bound iteration 2-step}
			\max\limits_ {1\leq i\leq n}  \frac{ \left| y_i - [T(T(x))]_i \right|}{y_i } \leq \varepsilon   \left( 1 - \frac{\delta}{b} \right)^2.
		\end{equation}
	\end{theorem}
	\begin{proof}
		Since $y$ is a fixed point, we have
		\[y_i = \delta + \sum\limits_{j=1}^n A_{ij} \frac{1}{y_j}.\]
		The first two iterations of $T$ applied to $x$ are
		\[[T(x)]_i = \delta + \sum\limits_{j=1}^n A_{ij} \frac{1}{x_j}, \quad\quad
		[T(T(x))]_i = \delta + \sum\limits_{j=1}^n A_{ij} \frac{1}{[T(x)]_j}.\]
		Let us now estimate the difference between $y$ and $T(x)$ on the $i$-th component
		\begin{equation}\label{eq:mainthm y-T(x) noabs}
			y_i - [T(x)]_i = \sum\limits_{j=1}^n A_{ij} \left(\frac{1}{y_j} - \frac{1}{x_j} \right) = \sum\limits_{j=1}^n A_{ij} \, \frac{x_j - y_j}{y_j x_j}.
		\end{equation}
		Taking absolute values yields
		\begin{equation}\label{eq:mainthm y-T(x) abs est}
			\left| y_i - [T(x)]_i \right| \leq \sum\limits_{j=1}^n A_{ij} \frac{1}{x_j} \, \frac{|x_j - y_j|}{y_j} \leq  
		\sum\limits_{j=1}^n A_{ij} \frac{1}{x_j} \varepsilon  = \varepsilon \left( [T(x)]_i - \delta \right). 
		\end{equation}
		Divide by $[T(x)]_i$ to get the relative distance of $y$ with respect to $T(x)$:
		\[ \frac{\left| y_i - [T(x)]_i \right|}{[T(x)]_i}\leq \varepsilon \left( 1 - \frac{\delta}{[T(x)]_i} \right). \]
		Since $x \in E$ and $T(x) \in E$, we have $ [T(x)]_i \leq b$, which gives 
		$1-\frac{\delta}{[T(x)]_i} \leq 1-\frac{\delta}{b}$. Therefore 
		\begin{equation}\label{eq:mainthm y-T(x) abs}
			\frac{\left| y_i - [T(x)]_i \right|}{[T(x)]_i}\leq \varepsilon \left( 1 - \frac{\delta}{b} \right). 
		\end{equation}
		In a similar manner, we estimate the distance between $y$ and $T^2(x)$
		\begin{equation}\label{eq:mainthm y-T2(x) noabs}
			y_i - [T(T(x))]_i = \sum\limits_{j=1}^n A_{ij} \left(\frac{1}{y_j} - \frac{1}{[T(x)]_j} \right) = \sum\limits_{j=1}^n A_{ij} \, \frac{[T(x)]_j - y_j}{y_j [T(x)]_j}.
		\end{equation}
		By taking absolute values and using \eqref{eq:mainthm y-T(x) abs}, we get
		\begin{equation*}
			\begin{aligned}
				\left| y_i - [T(T(x))]_i \right|  &\leq \sum\limits_{j=1}^n A_{ij} \frac{1}{y_j} \, \frac{\left| [T(x)]_j - y_j \right| }{[T(x)]_j}  \\ 
				& \leq \sum\limits_{j=1}^n A_{ij} \frac{1}{y_j} \varepsilon \left( 1 - \frac{\delta}{b} \right)   \\
				& = \left( y_i - \delta \right)	\varepsilon \left( 1 - \frac{\delta}{b} \right)  \\
				& = 	y_i \, \varepsilon \left( 1- \frac{\delta}{y_i} \right) \left( 1 - \frac{\delta}{b} \right)  \\
				& \leq 	y_i \, \varepsilon \left( 1 - \frac{\delta}{b} \right)^2. 
			\end{aligned}
		\end{equation*}
		Dividing by $y_i$, we get the relative distance of $T^2(x)$ with respect to $y$:
		\begin{equation*}
			\frac{ \left| y_i - [T(T(x))]_i \right|}{y_i } \leq \varepsilon \,  \left( 1 - \frac{\delta}{b} \right)^2.
		\end{equation*}
		The right-hand side is independent of $i$, giving us a uniform bound for all components.
	\end{proof}
	
	This theorem shows that iteration with the map $T$ results in linear convergence to the fixed point $y$. We have a convergence rate of $\left(1-\frac{\delta}{b}\right)^2$ for $T^2$, i.e., an average convergence rate of $1-\frac{\delta}{b}$ for $T$. More accurate bounds for the convergence rate can be obtained using the properties of the sets $E_k$. 
	\begin{corollary}\label{cor:iteration bound with E2}
		Let $x^{(0)}\in\mathbb{R}^n_+$ and $x^{(k+1)} = T(x^{(k)})$ for $k\geq 0$. Let $y$ be a fixed point of~$T$. Then 
		\[ \max\limits_{1\leq i\leq n} \, \frac{|y_i - x^{(2k+2)}_{i}|}{y_i} \leq
		\left(1 - \frac{\delta^2}{\delta^2+\dmax} \right)^{2k}
		\max\limits_{1\leq i\leq n} \, \frac{|y_i - x^{(2)}_{i}|}{y_i}. \]
		In particular, the sequence $x^{(k)}$ converges to $y$ for any choice of $x^{(0)}$. 
	\end{corollary}
	\begin{proof}
		We have seen that after two iterations, $x^{(2)} \in E = [\delta, \delta + \frac{\dmax}{\delta}]^n$; note that this implies $y\in E$. Applying Theorem~\ref{thm:error bound iteration 2-step} to $E$, where $b=\delta + \frac{\dmax}{\delta}$, yields the desired bound. 
	\end{proof}
	
	Another important consequence of Theorem~\ref{thm:error bound iteration 2-step} is the uniqueness of the fixed point of $T$. 
	\begin{corollary}\label{cor:uniqueness fixed point}
		The map $T$ admits a unique fixed point $y$.
	\end{corollary}
	\begin{proof}
		Consider a fixed point $y$, which exists by Proposition~\ref{prop:existence fixed point}.  Let $E =  [\delta, \delta + \frac{\dmax}{\delta}]^n$ as in the proof of Corollary~\ref{cor:iteration bound with E2}.  Suppose $z$ is another fixed point; then $z\in E$ as $z = T(T(z))$ and $T(T(z)) \in E$. Applying Theorem~\ref{thm:error bound iteration 2-step} with the set $E$ we obtain
		\[\max\limits_ {1\leq i\leq n} \frac{|y_i-z_i|}{y_i} = \max\limits_ {1\leq i\leq n} \frac{|y_i-[T(T(z))]_i|}{y_i}  \leq 
		\max\limits_ {1\leq i\leq n} \frac{|y_i-z_i|}{y_i}  \left( 1 - \frac{\delta^2}{\delta^2+\dmax} \right)^2. \]
		Since $\delta>0$, we have $1 - \frac{\delta^2}{\delta^2+\dmax} < 1$, so the equality is possible only if $y_i=z_i$ for all components, i.e., if $y=z$. 
	\end{proof}

	By Corollary~\ref{cor:limit of x^k is in E*}, the unique fixed point $y$ is in $E_*$. Therefore, for $k$ large enough $x^{(k)}$ is inside or arbitrarily close to $E_*$, so we can estimate the convergence rate using $b_*$ from Proposition~\ref{prop:limits of a_k and b_k} and its approximation \eqref{eq:a* and b* first order}.
	\begin{corollary}\label{cor:iteration bound with E*}
		An asymptotic upper bound for the convergence rate of $x^{(k)}$ to the fixed point $y$ is 
		\begin{equation}\label{eq:asymptotic_errbound}
			1-\frac{\delta}{b_*} = 1-\frac{\delta^2}{\dmax-\dmin} + O(\delta^4) \quad\;\; \mathrm{for}\; \delta \to 0.
		\end{equation}
	\end{corollary}
	
	We put together all the results in the following theorem. 
	
	\begin{theorem}[Main result]\label{thm:main}
		Let $G$ be a (possibly directed and weighted) graph with minimum out-degree $\dmin$ and maximum out-degree $\dmax$. Let $T$ be the Fitness Centrality iteration map for $\delta > 0$. Then
		\begin{itemize}
			\item There exists a unique fixed point $y$ such that $T(y) = y$, i.e., the Fitness Centrality vector is well-defined.
			\item  The point $y$ belongs to $E_* = [a_*, b_*]^n$ given by Proposition~\ref{prop:limits of a_k and b_k}.
			\item Starting from any initial condition $x^{(0)} \in \mathbb{R}^n_+$, the sequence of standard iterates $x^{(k)}$ converges linearly to $y$.
			\item An upper bound for the convergence rate valid from $x^{(2)}$ onwards is 
			$1-\frac{\delta^2}{\delta^2+\dmax}$. 
			\item The asymptotic upper bound for the convergence rate is $1-\frac{\delta}{b_*}$, which is approximately $1-\frac{\delta^2}{\dmax-\dmin}$ for small $\delta$. 
		\end{itemize}
	\end{theorem}
	
	Note that the positivity of $\delta$ is crucial for the bounds provided in Theorem~\ref{thm:error bound iteration 2-step} and its corollaries. If $\delta=0$ were allowed, then \eqref{eq:error bound iteration 2-step} would only guarantee that the relative error does not increase between iterates. Indeed, numerical experiments show that the Fitness Centrality iterations fail to converge for $\delta=0$.
	
	\section{Convergence acceleration} \label{sec:comp}
	The convergence rate proven in the previous section is $1-\frac{\delta}{b}$, which is close to 1 when $\delta$ is small and $b$ is moderately large. Consequently, although the standard iteration converges linearly, its rate can be exceedingly slow in practice, as confirmed by the numerical experiments of Section~\ref{sec:num_exp}.
	In this section, we introduce alternative methods to compute a new iteration step; here, as well as in Section~\ref{sec:num_exp}, we indicate by $\{x^{(k)}\}_{k \in \mathbb{N}}$ any sequence of iterates obtained using any combination of the available steps described below. We use a subscript when we need to specify which method has been applied in the last step; for example, the standard iteration step is denoted as  
	\[x^{(k+1)}_{\mathrm{std}} = T(x^{(k)}).\]
	
	The first two approaches are based on linear acceleration: the new point is obtained as a linear combination of $m$ points (typically previous iterates).
	The first technique uses $m=2$ points; given $x^{(k)}$, it constructs the new point as the average of $x^{(k)}$ and $T(x^{(k)})$, whereas the second is
	 \emph{Anderson acceleration}, which constructs the new point as a linear combination of the previous $m$ iterates with appropriately chosen weights. 
	
	The third and most important technique is \emph{Newton's method}, which has a quadratic rate of convergence in the neighborhood of the solution, and can significantly speed up the final stages of the iterative process. 
	
	The fourth technique is gradient descent, which was introduced in \cite{CostFunctions2026} for computing the Fitness Centrality of undirected graphs.
	
	\subsection{Linear acceleration with two term averaging}
	In numerical experiments, one frequently observes that the sequence of standard iterates $x^{(k)}_{\mathrm{std}}$ converges in an oscillating manner to $y$ depending on the parity of $k$ (i.e., they converge from below for odd $k$ and from above for even $k$, or vice versa). This behavior has a heuristic interpretation: if the components of $x$ are small, then the components of $\inv(x)$ are large, and therefore those of $T(x)=\delta \one + A \cdot \inv(x)$ are also large. Then, the components of $\inv(T(x))$ are small, and so are the components of $T^2(x)$. The following lemma provides a sufficient condition for the oscillating convergence. 
	
	\begin{lemma}
		Let $y$ be the fixed point of $T$ and let $x\in \mathbb{R}^n_+$ be such that $y_i \geq x_i$ for all components $i$. Then $y_i \leq [T(x)]_i $ for all components $i$.  
		If on the other hand $y_i \leq x_i$ for all components $i$,  then $y_i \geq [T(x)]_i $ for all components $i$.  
	\end{lemma}
	\begin{proof}
		We have
		\[y_i = [T(y)]_i = \delta + \sum\limits_{j=1}^n A_{ij} \frac{1}{y_j}, \hspace{5mm}
		[T(x)]_i = \delta + \sum\limits_{j=1}^n A_{ij} \frac{1}{x_j}.\]
		From $y_j \geq x_j$  we get $\frac{1}{y_j} \leq \frac{1}{x_j}$, and comparing term by term we get the desired result. The proof is the same in the case when $y_i \leq x_i$. 
	\end{proof}
	
	Since $x_i \leq y_i \leq [T(x)]_i$, the average $\frac{1}{2}\left(x_i + [T(x)]_i\right)$ reduces the maximum distance to $y_i$ compared to the two endpoints $x_i$ and $[T(x)]_i$. 
	Because $T$ swaps large and small components, often the same ordering between $x_i$, $y_i$, and $[T(x)]_i$ holds  even without the assumption that all components of $x$ are smaller (or larger) than those of $y$.  Therefore, for any iterate $x^{(k)}$, it is reasonable to consider the average of $x^{(k)}$ and $T(x^{(k)})$.
	
	\begin{definition}\label{def:averaged_step}
		The \emph{averaged step} is defined as $x^{(k+1)}_{\mathrm{avg}} = \frac{1}{2} \left(x^{(k)} + T(x^{(k)}) \right)$.
	\end{definition}
	Observe that, starting from $x^{(k)}$, the standard iteration with the map $T$ generates the sequence $(x^{(k)}, T(x^{(k)}), T^2(x^{(k)}), \ldots)$; the average of the first two elements is precisely $x^{(k+1)}_{\mathrm{avg}}$. Using the averaged step, we replace the second element of the sequence with $x^{(k+1)}_{\mathrm{avg}}$, at essentially the same computational cost as evaluating $T(x^{(k)})$.
	
	Numerical tests suggest that instead of exclusively employing averaged steps, it is better to alternate between one averaged step and one standard step. We evaluate this strategy and compare it to other possible strategies in Section~\ref{sec:num_exp}. For the sake of clarity, we highlight that the resulting sequence with the alternating scheme is
	\[x^{(k-1)}, \; x^{(k)}_{\mathrm{std}} = T(x^{(k-1)}),\;  x^{(k+1)}_{\mathrm{avg}}= \frac{1}{2} \left(x^{(k)}_{\mathrm{std}} + T(x^{(k)}_{\mathrm{std}})\right),\;   x^{(k+2)}_{\mathrm{std}} =T(x^{(k+1)}_{\mathrm{avg}}),\; \ldots \,.  \]
	Under the assumption that $x^{(k)} = T(x^{(k-1)})$, we estimate the distance between $x^{(k+1)}_{\mathrm{avg}}$ and the fixed point $y$. Suppose that $E = [a,b]^n$ is a set satisfying $T(E) \subseteq E$ and $x^{(k-1)}\in E$.
	
	First, suppose that $x^{(k-1)}_i \leq y_i$ for all $i$. Setting $x=x^{(k-1)}$ in Equations~\eqref{eq:mainthm y-T(x) noabs} and~\eqref{eq:mainthm y-T2(x) noabs} and adding them, we obtain
	\begin{equation*}
		\begin{aligned}
			y_i - \frac{1}{2} \left(x^{(k)}_i + [T(x^{(k)})]_i\right) &=  \sum\limits_{j=1}^n \frac{A_{ij}}{y_j} \cdot \frac{1}{2} \left(\frac{x^{(k-1)}_j - y_j}{x^{(k-1)}_j} + \frac{x^{(k)}_j - y_j}{x^{(k)}_j} \right).
		\end{aligned}
	\end{equation*}
	Since $x^{(k-1)}_j - y_j$ and $x^{(k)}_j - y_j$ have different signs, the sum is smaller in absolute value than either of them, leading to a convergence rate smaller than $1-\frac{\delta}{b}$.  
	
	Note that if neither condition ``$x^{(k)}_i \leq y_i$ for all $i$'' nor ``$x^{(k)}_i \geq y_i$ for all $i$'' holds, after an averaged step we no longer know whether $x^{(k+1)}_{\mathrm{avg}}$ has components larger or smaller than those of $y$. In general, we have
	\begin{equation*}
		\begin{aligned}
			\left| y_i - \frac{1}{2} \left(x^{(k)}_i + [T(x^{(k)})]_i\right) \right| &\leq \frac{1}{2} \left| y_i - x^{(k)}_i \right| + \frac{1}{2}  \left| y_i - [T(x^{(k)})]_i \right| \\
			& \leq \max\left(  \left| y_i - x^{(k)}_i \right|, \left| y_i - [T(x^{(k)})]_i \right|   \right) \\
			& = \max\left(  \left| y_i - [T(x^{(k-1)})]_i \right|, \left| y_i - [T^2(x^{(k-1)})]_i \right|   \right)
		\end{aligned}
	\end{equation*}
	which means that, in the worst case, a combination of a standard step and an averaged step gives an absolute error no greater than the maximum of the absolute errors after one and two standard steps.

	\subsection{Anderson acceleration}
	The approach described in the previous section constructs the new iterate $x^{(k+1)}_{\mathrm{avg}}$ as a linear combination of the two previous (standard) iterates with weights $\frac{1}{2}$. This idea can be extended by considering a linear combination of the previous $m$ iterates
	\begin{equation}\label{eq:linear_acceleration_general_x}
		x^{(k+1)}_{\mathrm{acc}} = \alpha_1 x^{(k)} + \alpha_2 x^{(k-1)} + \ldots + \alpha_m x^{(k-m+1)},
	\end{equation}
	for some coefficients $\alpha_l$. This technique is known as \emph{acceleration} for fixed-point iterations. For a thorough review of acceleration approaches, see \cite{Saad2025}. 
	
	What fundamentally characterizes an acceleration method is the strategy employed to select these coefficients. A necessary condition, called the \emph{consistency condition}, is that  $\sum_{l=1}^m \alpha_l = 1$. Indeed, suppose that all $m$ previous iterates are close to the fixed point $y$. Then, the right-hand side of \eqref{eq:linear_acceleration_general_x} is approximately $(\alpha_1 + \ldots + \alpha_m) y$; in order not to stray farther away from the fixed point, we require that the coefficients sum to~1.
	
	 One of the most widely used methods is \emph{Anderson acceleration} \cite{Anderson1965}; for a comparison of this approach to other acceleration techniques, see~\cite{BrezinskiRedivoSaad2018}. We now describe how the coefficients in Anderson acceleration are chosen. Consider the function $f: \mathbb{R}_+^n\to \mathbb{R}^n$, defined as
	\begin{equation}\label{eq:f-residue_definition}
		f(x) = x - T(x).
	\end{equation}
	Note that $f(x)=0$ if and only if $x$ is the fixed point of $T$; in general, the magnitude of $f$ is an indicator of how close $x$ is to being the fixed point. 
	Letting  $f^{(k)} = f(x^{(k)})$ denote the residual at step $k$, consider the linear combination of residuals using the same coefficients as in \eqref{eq:linear_acceleration_general_x}:
	\begin{equation}\label{eq:linear_acceleration_general_f}
		\widetilde{f}^{(k+1)} = \alpha_1 f^{(k)} + \alpha_2 f^{(k-1)} + \ldots + \alpha_m f^{(k-m+1)}.
	\end{equation}
	This quantity can be interpreted as the linearized residual of $f$ at  $x^{(k+1)}$; that is, $\widetilde{f}^{(k+1)}$ is approximately $f^{(k+1)} = f(x^{(k+1)}_{\mathrm{acc}})$.
	The rationale behind Anderson acceleration is to select the coefficients $\alpha_l$ that minimize the 2-norm of the right-hand side of \eqref{eq:linear_acceleration_general_f}.
	
	\begin{definition}\label{def:Anderson_step}
		Given $m$ iterates $x^{(k-m+1)}, \ldots, x^{(k)}$, the \emph{Anderson step} is 
		\begin{equation*}
			x^{(k+1)}_{\mathrm{And}} = \alpha_1 x^{(k)} + \alpha_2 x^{(k-1)} + \ldots + \alpha_m x^{(k-m+1)},
		\end{equation*}
		where the coefficients are obtained as 
		\begin{equation*}
			(\alpha_1, \ldots, \alpha_m) =  \arg\min\left\{
			\left\| \sum\limits_{l=1}^m \alpha_l f^{(k-l+1)} \right\|_2 \text{ s.t. } \sum\limits_{l=1}^m \alpha_l=1
			\right\}.
		\end{equation*}
	\end{definition}
	In practice, to apply one Anderson step, we need to solve a least-squares problem with a tall rectangular $n\times m$ matrix, subject to a linear constraint. While using Anderson acceleration reduces the number of iterations needed, it is often more advantageous not to use Anderson steps exclusively, but to interleave them with ``simple'' steps. For the Fitness Centrality iteration, numerical evidence suggests using either standard iterations with the map $T$ or averaged steps (Definition~\ref{def:averaged_step}) as simple steps.
	
	\subsection{Newton's method}
	The averaged and Anderson steps accelerate convergence, but they still yield linear convergence rates to the solution. In this section we describe Newton's method \cite[Ch. 3]{Fletcher2000}, which instead exhibits quadratic convergence. 
	
	We are interested in finding the zero of the function $f$ from \eqref{eq:f-residue_definition}, as it coincides with the fixed point of $T$. Let $J_f(x)$ denote the Jacobian matrix of $f$ at $x$; the expression for its $ij$-th component is
	\[[J_f(x)]_{ij} = \frac{\partial [f(x)]_i}{\partial x_j} = \delta^{\mathrm{kron}}_{ij} + A_{ij} \frac{1}{x_j^2},\]
	where $\delta^{\mathrm{kron}}_{ij}$ denotes the Kronecker delta of $i$ and $j$. Denoting by $D_x$ the diagonal matrix with $i$-th diagonal entry equal to $x_i$, the Jacobian can be equivalently written in matrix notation as
	\[J_f(x) =  I + A D_x^{-2}.\]
	
	A step of Newton's method is obtained by solving a linear system with the Jacobian. 
	
	\begin{definition}\label{def:Newton_step}
		The \emph{Newton step} is defined as 
		\[x^{(k+1)}_{\mathrm{New}} = x^{(k)} - J_f(x^{(k)})^{-1} \, f(x^{(k)}).\]
	\end{definition}
	
	Newton's method utilizes first-order information on $f$, namely the derivatives present in $J_f(x)$, generally achieving quadratic convergence and requiring far fewer steps than  any of the previously described methods.

	The main drawback is that Newton's method is computationally more expensive, as each step requires the solution of a linear system of size $n$. Without any assumption on the graph, the cost of solving the linear system
	\begin{equation}\label{eq:Newton linear system solve}
	    B z = f(x^{(k)}), \quad \text{ with }\quad  B=J_f(x^{(k)})=I+AD^{-2}_{x^{(k)}},
	\end{equation}
	is $O(n^3)$, for example using a direct method such as the LU factorization. However, most large-scale networks in applications are sparse, meaning that $A$ has $O(n)$ nonzero entries; therefore $B$ is also sparse. In this case, a direct sparse solver can be employed; the computational cost depends heavily on the graph, as it may vary from $O(n)$ (for example, for a path graph with a tridiagonal adjacency matrix) to  $O(n^3)$. See \cite{ScoTum2023} for details on solving sparse linear systems. 
	If $n$ is sufficiently large to make the application of direct methods prohibitive, an iterative method may be used \cite{Saad2003}.
	In this case, the rate of convergence of the method will generally depend on the conditioning and spectral properties of the Jacobian matrix obtained at each Newton
	step.  Preconditioning techniques can be used to obtain rapid convergence in many cases \cite{Benzi2002}, 
	but the actual computational cost of each Newton iteration is highly problem-dependent.
	For most networks, the cost will be somewhere between $O(n)$ and $O(n^2)$.
	
	Another limitation is that convergence is guaranteed only in a neighborhood of the solution: using Newton's method when the distance to the solution is too large may cause the sequence to fail to converge. Furthermore, the Fitness Centrality iteration is defined only for vectors $x\in \mathbb{R}^n_+$; in general, after a Newton step, the iterate is not guaranteed to remain in the positive orthant. While projection techniques may be employed to guarantee positivity of the vector components, we cannot project onto the boundary of $\mathbb{R}^n_+$ since some of the components would be zero. Therefore, a projection technique should choose a point not on the boundary of $\mathbb{R}^n_+$. It should be sufficiently close to the boundary in order to be approximately a projection; however, it should not be too close in order to avoid a blow-up in the subsequent Newton iterations. Therefore, it is highly non-trivial how this choice should be made.
	
	All things considered, we recommend using a stationary method for the initial iterations (for example, a combination of Anderson acceleration and standard or averaged steps) and then switching to Newton's method once the iterates are sufficiently close to the fixed point. If Newton's method fails to converge, or if an iterate outside the admissible region $\mathbb{R}^n_+$ is produced, then the tolerance for switching to Newton's method should be made smaller.
	
	\subsection{Gradient descent}
	We describe the gradient descent approach, originally proposed in \cite{CostFunctions2026}. 
	Consider an undirected graph, meaning that $A=A^\top$. Under the change of variables $h: \mathbb{R}_+^n \to \mathbb{R}^n$, defined component-wise by $z_i = [h(x)]_i = -\log(x_i)$, the continuous version of the dynamical system obtained from \eqref{eq:FC iteration} is now integrable and admits a potential given by
	\[\widetilde{U}: \mathbb{R}^n \to \mathbb{R}, \quad
	\widetilde{U}(z) = \frac{1}{2} \sum\limits_{i=1}^n  \sum\limits_{j=1}^n A_{ij} e^{z_i} e^{z_j} + \delta  \sum\limits_{i=1}^n e^{z_i} - \sum\limits_{i=1}^n z_i.\]
	In the $x$-domain, the potential becomes
	\begin{equation*}
		\begin{aligned}
			U:  \mathbb{R}_+^n \to \mathbb{R}, \quad 
			U(x) &= \frac{1}{2} \sum\limits_{i=1}^n  \sum\limits_{j=1}^n A_{ij} \frac{1}{x_i} \frac{1}{x_j}  + \delta \sum\limits_{i=1}^n \frac{1}{x_i} - \sum\limits_{i=1}^n \log\left(\frac{1}{x_i}\right) \\
			&=\frac{1}{2} \inv(x)^\top A \inv(x) + \delta\, \one^\top \inv(x) - \one^\top \log\inv(x),
		\end{aligned}
	\end{equation*}
	where the logarithm is taken entrywise. In \cite{CostFunctions2026}, the authors prove that the potential is strictly convex, and therefore admits a unique minimum, which is precisely the fixed point of $T$. Therefore, one can apply the gradient descent algorithm to find this minimum. The gradient is, component-wise, 
	\[ [\nabla \widetilde{U}(z)]_i = \sum\limits_{j=1}^n A_{ij} e^{z_j} e^{z_i} + \delta\, e^{z_i} -1.\]
	\begin{definition}\label{def:gradient_step}
		The \emph{gradient step} in the $z$-domain is defined as 
		\[z^{(k+1)}_{\mathrm{grd}} = z^{(k)} - \gamma_k \, \nabla \widetilde{U}(z^{(k)}),\]
		where $\gamma_k$ is a step size found by backtracking, ensuring that $\widetilde{U}(z^{(k+1)}) < \widetilde{U}(z^{(k)})$.
		The corresponding vector in the $x$-domain has components $x^{(k+1)}_{\mathrm{grd},i} = \exp(-z^{(k+1)}_{\mathrm{grd},i})$.
	\end{definition}
	Naturally, once the objective is reformulated as a convex minimization problem, any of the many available optimization methods can be used; see, for example, \cite{BoydVandenberghe2004, Fletcher2000}.

	\section{Numerical experiments} \label{sec:num_exp}
	Having seen the multiple approaches that can be used to speed up convergence, we now compare their performance  on real-world networks. 
	
	To measure how close a point $x$ is to being a fixed point, we employ two metrics:
	\begin{itemize}
		\item the (relative) backward error, $\errT(x) \coloneqq\max\limits_{1\leq i\leq n} \, \dfrac{|x_i - [T(x)]_i|}{x_i}$; \vspace{0.2ex}
		\item the (relative) forward error, $\erry(x) \coloneqq \max\limits_{1\leq i\leq n} \, \dfrac{|x_i - y_i|}{y_i}$.
	\end{itemize}

The backward error can be computed on the fly at each iteration $x^{(k)}$; this makes it the ideal (and only) candidate for being the stopping criterion: when $\errT(x^{(k)})$ falls below a certain tolerance \texttt{tol}, the algorithm terminates. 
	
	We note that the numerators present in the backward error at the $k$-th step, $\errT(x^{(k)})$, are the absolute values of the components of the residual $f^{(k)}$. Furthermore, when standard steps are used, the backward error is the relative distance between two consecutive iterates; when Anderson acceleration is employed, computing the residual $f^{(k)}$ is already required to obtain the Anderson step. In these two cases, calculating the backward error does not incur any additional computational cost. 
	
	On the other hand, the forward error is not (exactly) computable since the fixed point $y$ is not known, and therefore it cannot be used as a stopping criterion.
	In order to assess the accuracy of the approximations produced by an iterative procedure we can estimate the forward error by using a reference
	solution $\hat y\approx y$ obtained, for instance, by performing additional iterations and working in higher precision
	(details are provided below). This is also useful to compare the accuracy of the iterates with 
	the theoretical results provided by Theorem~\ref{thm:error bound iteration 2-step}. 
		
	The possible iterative steps that can be used are the following: 
	\begin{enumerate}[itemsep=1.2ex,label=\alph*)]
		\item $x^{(k+1)}_{\mathrm{std}} = T(x^{(k)})$, standard step (Definition~\ref{def:FC_step});
		\item $x^{(k+1)}_{\mathrm{avg}}= \frac{1}{2} \left(x^{(k)} + T(x^{(k)}) \right)$, averaged step (Definition~\ref{def:averaged_step}); 
		\item $x^{(k+1)}_{\mathrm{And}}=\sum\limits_{l=1}^m \alpha_l \, x^{(k+1-l)}$, Anderson step (Definition~\ref{def:Anderson_step});
		\item $x^{(k+1)}_{\mathrm{New}}= x^{(k)} - J_f(x^{(k)})^{-1} \, f(x^{(k)})$, Newton step (Definition~\ref{def:Newton_step});
		\item $x^{(k+1)}_{\mathrm{grd}}$, gradient descent step, valid only for undirected graphs (Definition~\ref{def:gradient_step}).
	\end{enumerate}
	
	We employ different strategies that combine the above steps. Indeed, Newton's method can be applied only in the proximity of the solution, and Anderson acceleration works best if interlaced with standard steps. Furthermore, all methods benefit from applying some standard steps at the outset of the iteration. 
	
	The algorithm continues iterating until the backward error reaches a prescribed tolerance \texttt{tol}. If Newton's method is desired, then the strategy switches to the Newton step if the backward error is below a tolerance \texttt{tol\_Newton}. 
	
	We compare the following strategies:
	\begin{enumerate}
		\item \textit{T-map strategy}: always use standard steps until convergence. 
		\item \textit{Averaged strategy}:
		\begin{itemize}
			\item do 2 standard steps;
			\item alternate between standard and averaged steps until convergence;
			\item \textit{optional:} if the current error is below  \texttt{tol\_Newton}, use Newton steps instead. 
		\end{itemize}
		\item \textit{Anderson strategy}: 
		\begin{itemize}
			\item choose parameters $m$ (window size) and $q$ (how many standard steps before an Anderson step);
			\item do 2 standard steps;
			\item do $q$ standard steps, followed by an Anderson step on the previous $m$ iterates; repeat until convergence;
			\item \textit{optional:} if the current error is below  \texttt{tol\_Newton}, use Newton steps instead. 
		\end{itemize}
		\item \textit{Gradient strategy}:
		\begin{itemize}
			\item do 6 standard steps;
			\item use gradient steps until convergence.
		\end{itemize}
		
		\medskip 
		
	Indeed, we found that performing 6 standard steps before activating the gradient strategy promotes stability of the overall iterative process.
	\end{enumerate}

	We stress that, while the error bounds guarantee convergence from the theoretical point of view, the above computational strategies are heuristic. Other acceleration techniques, or a different combination of approaches, may also yield good performance.
	
	The Anderson and averaged approaches are presented both with and without the use of Newton's method. This leads to six different strategies in total, for which we plot the backward and forward errors. For the forward error, the 
  ``exact'' solution $\hat y$ we use for reference is obtained as follows: first, the solution is computed with the averaged+Newton strategy in double-precision arithmetic; then, additional iterations (two Newton steps for Figures~\ref{fig:exp_celegans_fwd} and~\ref{fig:exp_atc_fwd}, and six standard steps for Figure~\ref{fig:exp_BACI96_fwd}) are computed in quadruple-precision arithmetic. While this may seem to favor the averaged strategy, we stress that the other methods exhibit convergence to the reference solution too. Furthermore, the backward error can be used as a safety check to verify how close both the obtained numerical solution and the computed iterates are to being the true fixed point. In all plots, there is a clear correlation between the backward and forward errors. 
	
	We present three real-world networks on which we test the strategies. We set the tolerance of the Newton switch to \texttt{tol\_Newton}~$=10^{-2}$ and the tolerance for terminating the algorithm to \texttt{tol}~$=10^{-10}$.  The definition of the map $T$ depends on $\delta$, which can be seen as a regularization parameter. When $\delta$ is small, we are closer to the ideal solution (that is, with $\delta=0$); however, the iterations of Fitness Centrality for a generic graph usually fail to converge for $\delta=0$. Furthermore, excessively small values of $\delta$ lead to slow convergence. In  \cite{regularizedEFC2018}, the authors analyze the components of the obtained solution, finding that for values of $\delta$ between $10^{-6}$ and $10^{-1}$, the entries of the solution are stable. We found that values of $\delta$ between $10^{-3}$ and $10^{-2}$ provide a good balance between stability and convergence speed; in the experiments we set $\delta=10^{-2}$. We use $x^{(0)} = 2 \cdot \one$ as the starting value; we have avoided using $x^{(0)}=\one$ since it leads to the initial value of $z^{(0)}=0$ in the $z$-domain for the gradient step.
	
	For the Anderson strategy, we use a window size $m=4$ and perform an Anderson step every $q=5$ standard steps. We also plot the asymptotic theoretical bound for the forward error (Corollaries~\ref{cor:iteration bound with E2} and~\ref{cor:iteration bound with E*}) starting from the iteration $k_0=20$, which is
	\[\text{upper bound at } k \text{-th step} = \left(1-\frac{\delta}{b_*}\right)^{k-k_0} \cdot \erry(x^{(k_0)}).\] 
	
	\paragraph{\textit{C. elegans} brain network}
	We start with an undirected and weighted graph with $n=297$ nodes, representing the neural connections of \textit{Caenorhabditis elegans}, originally published in \cite{Celegans1986}. The graph data\footnote{Available at \url{http://konect.cc/networks/dimacs10-celegansneural/}} 
	is taken from the Konect Database \cite{konect}. The plots of the errors for the different strategies are represented in Figure~\ref{fig:exp_celegans}.
	
Without using Newton steps, the strategy with the highest convergence rate is the Anderson strategy, followed by the gradient strategy, which is slightly faster than the averaged strategy. The slowest one is T-map, exhibiting stagnating behavior. The error bound, although quite close to the forward errors of the T-map strategy, has limited practical utility: its convergence rate (Corollary~\ref{cor:iteration bound with E*}) is $1-\frac{\delta}{b_*} \approx 1 - 6\cdot 10^{-8}$, so the upper bound remains almost constant.
The discrepancy between the theoretical bound and the actual convergence rates can be justified as follows. In the proof of Theorem~\ref{thm:error bound iteration 2-step}, the upper bound of \eqref{eq:mainthm y-T(x) abs est} achieves equality if the components of an iterate $x$ are all greater (or all smaller) than the components of the fixed point $y$. Furthermore, we bound the absolute value of $y_i$ and $[T(x)]_i$ with $b$, the rightmost bound of the box $E=[a,b]^n$. This inequality is far from sharp, as $b$ is two orders of magnitude greater than the components of the fixed point $y$ in our experiments.

	On the other hand, switching to Newton's method significantly improves the convergence speed,
	lowering the error from $10^{-2}$ to below $10^{-13}$ in just three (Anderson+Newton strategy) or four (averaged+Newton strategy) iterations. The small number of iterations needed demonstrates the strength of a second-order method compared to the stationary methods, which exhibit linear convergence as can be seen from the plots.

	%%% C. elegans brain network, undirected, weighted
	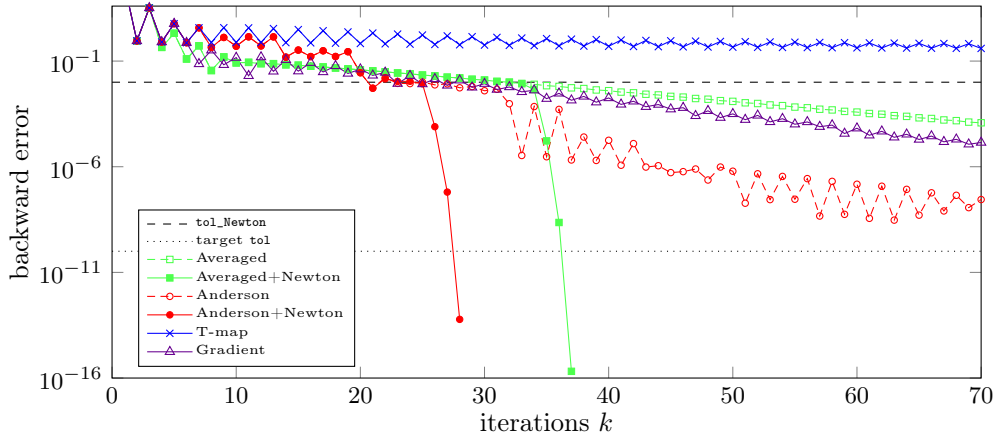
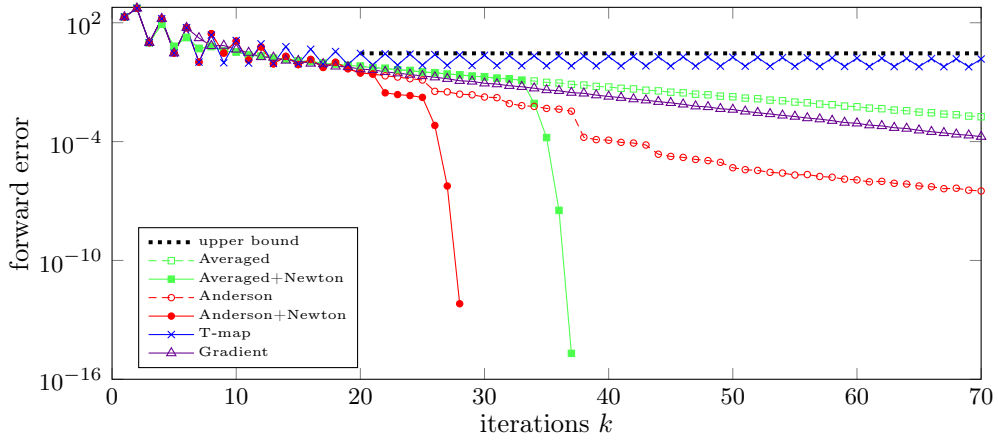
\begin{figure}
		\def\pathtodir{dataset/celegans-w}
		\begin{subfigure}{1\textwidth}
			\begin{tikzpicture}
				\begin{semilogyaxis}[
					xlabel={iterations $k$},
					ylabel={backward error},
					width=\textwidth, height=6.5cm,
					ymin=10^-16, ymax=40,
					xmin=0,xmax=70,
					xlabel style={yshift=2mm, },
					xticklabel style={font=\small,},
					yticklabel style={font=\small,/pgf/number format/fixed,},
					legend pos = south west,
					legend cell align={left},
					legend style={font=\tiny, nodes={inner sep=1pt}},
					]
					
					\addplot[mark=none, dashed, thin, black, samples=2, domain=0:100] {1e-2}; \label{tolN}% horizontal line			
					\addplot[mark=none, dotted, thin, black, samples=2, domain=0:100] {1e-10}; \label{tol}% horizontal line			
					
					\addplot [color=green!70, mark = square, mark size = 1.2pt, mark options={solid}, densely dashed] table {\pathtodir/err_iter_2avgNN.txt};
					\addplot [color=green!70, mark = square*, mark size = 1.2pt] table {\pathtodir/err_iter_2avg.txt};
					\addplot [color=red, densely dashed, mark = o, mark size = 1.2pt,  mark options={solid}, densely dashed] table {\pathtodir/err_iter_3AndNN.txt};
					\addplot [color=red, mark = *, mark size = 1.2pt] table {\pathtodir/err_iter_3And.txt};
					\addplot [color=blue, mark = x, mark size = 1.9pt] table {\pathtodir/err_iter_1Tmap.txt};
					\addplot [color=purpleG, mark = triangle] table {\pathtodir/err_iter_4grad.txt};
					
					\addlegendentry{\texttt{tol\_Newton}}
					\addlegendentry{target \texttt{tol}}
					\addlegendentry{Averaged}
					\addlegendentry{Averaged+Newton}
					\addlegendentry{Anderson}
					\addlegendentry{Anderson+Newton}
					\addlegendentry{T-map}
					\addlegendentry{Gradient}
				\end{semilogyaxis}
			\end{tikzpicture}	
			\caption{Backward error  $\max\, \{|x_i^{(k)}-[T(x^{(k)})]_i| / |x_i^{(k)}|$ for $1\leq i \leq n\}$}
			\label{fig:exp_celegans_bwd}
		\end{subfigure}%
		%\hfill%	
		
		\begin{subfigure}{1\textwidth}%
			\begin{tikzpicture}
				\begin{semilogyaxis}[
					xlabel={iterations $k$},
					ylabel={forward error},
					width=\textwidth,height=6.5cm,
					ymin=10^-16,ymax=600,
					xmin=0,xmax=70,
					xlabel style={yshift=2mm, },
					xticklabel style={font=\small,},
					yticklabel style={font=\small,/pgf/number format/fixed,},
					legend pos = south west,
					legend cell align={left},
					legend style={font=\tiny, nodes={inner sep=1pt}},
					]
					
					\addplot [color=black, mark = none, dotted, line width = 1.5pt] table {\pathtodir/err_bounds.txt};
					\addplot [color=green!70, mark = square, mark size = 1.2pt, mark options={solid}, densely dashed] table {\pathtodir/err_sol_2avgNN.txt};
					\addplot [color=green!70, mark = square*, mark size = 1.2pt] table {\pathtodir/err_sol_2avg.txt};
					\addplot [color=red, densely dashed, mark = o, mark size = 1.2pt,  mark options={solid}, densely dashed] table {\pathtodir/err_sol_3AndNN.txt}; 
					\addplot [color=red, mark = *, mark size = 1.2pt] table {\pathtodir/err_sol_3And.txt};
					\addplot [color=blue, mark = x, mark size = 1.9pt] table {\pathtodir/err_sol_1Tmap.txt}; 
					\addplot [color=purpleG, mark = triangle] table {\pathtodir/err_sol_4grad.txt}; 
					
					\addlegendentry{upper bound}
					\addlegendentry{Averaged}
					\addlegendentry{Averaged+Newton}
					\addlegendentry{Anderson}
					\addlegendentry{Anderson+Newton}
					\addlegendentry{T-map}
					\addlegendentry{Gradient}			
				\end{semilogyaxis}
			\end{tikzpicture}	
			\caption{Forward error $\max\, \{|x_i^{(k)}-y_i| / |y_i|$ for $1\leq i \leq n\}$}
			\label{fig:exp_celegans_fwd}
		\end{subfigure}%
		\caption{Error comparison of six different strategies for computing Fitness Centrality on the \textit{C. elegans} brain network (undirected, weighted)}
		\label{fig:exp_celegans}
	\end{figure}

	\paragraph{Air Traffic Control}
	In Figure~\ref{fig:exp_atc} we show the errors of the second example, which is a directed and unweighted network obtained from the Preferred Route Database of the US National Flight Data Center. Nodes are airports or service centers; a directed edge from $i$ to $j$ is present if it represents a segment of a preferred route. The graph data\footnote{Available at \url{http://konect.cc/networks/maayan-faa/}} 
	is taken from the Konect Database \cite{konect}; we restrict our analysis to the subgraph formed by the largest strongly connected component, which contains 792 vertices.
	
	The convergence of the three stationary methods stagnates, as can be seen from both plots. The graph is directed, so it falls outside the framework in which the gradient method was developed;  indeed, the gradient strategy fails to converge if applied. 
	In this experiment, incorporating Newton's method is crucial for achieving convergence, as the stationary methods alone would require an excessively high number of steps to reach the solution.
	
	%%% Air Traffic Control network, directed
	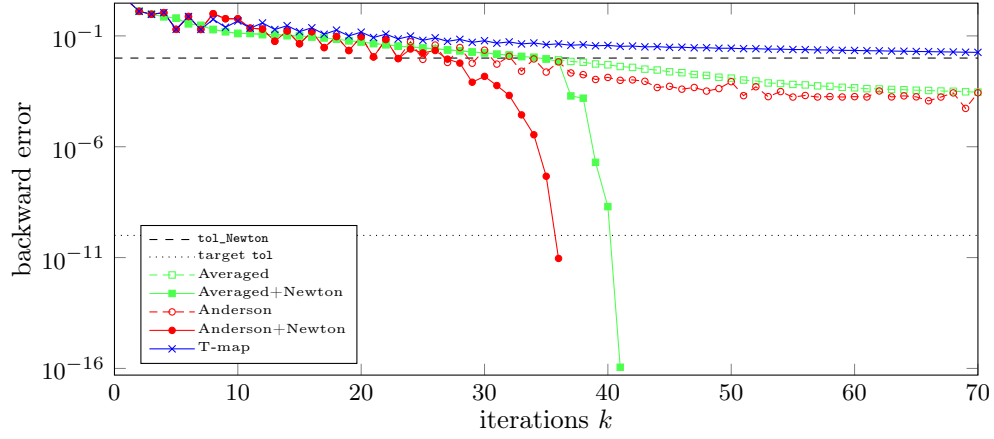
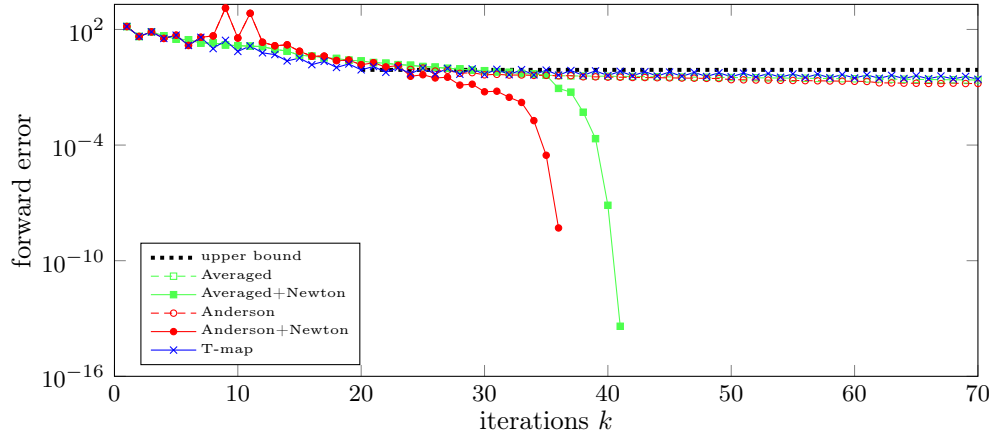
\begin{figure}
		\def\pathtodir{dataset/atc}
		\begin{subfigure}{1\textwidth}
			\begin{tikzpicture}
				\begin{semilogyaxis}[
					xlabel={iterations $k$},
					ylabel={backward error},
					width=13cm, height=6.5cm,
					ymin=0.5*10^-16, ymax=3,
					xmin=0,xmax=70,
					xlabel style={yshift=2mm, },
					xticklabel style={font=\small,},
					yticklabel style={font=\small,/pgf/number format/fixed,},
					legend pos = south west,
					legend cell align={left},
					legend style={font=\tiny, nodes={inner sep=1pt}},
					]
					
					\addplot[mark=none, dashed, thin, black, samples=2, domain=0:100] {1e-2}; \label{tolN}% horizontal line			
					\addplot[mark=none, dotted, thin, black, samples=2, domain=0:100] {1e-10}; \label{tol}% horizontal line			
										
					\addplot [color=green!70, mark = square, mark size = 1.2pt, mark options={solid}, densely dashed] table {\pathtodir/err_iter_2avgNN.txt};
					\addplot [color=green!70, mark = square*, mark size = 1.2pt] table {\pathtodir/err_iter_2avg.txt};
					\addplot [color=red, densely dashed, mark = o, mark size = 1.2pt,  mark options={solid}, densely dashed] table {\pathtodir/err_iter_3AndNN.txt};
					\addplot [color=red, mark = *, mark size = 1.2pt] table {\pathtodir/err_iter_3And.txt};
					\addplot [color=blue, mark = x, mark size = 1.9pt] table {\pathtodir/err_iter_1Tmap.txt};
					%				\addplot [color=purpleG, mark = triangle] table {\pathtodir/err_iter_4grad.txt};
					
					\addlegendentry{\texttt{tol\_Newton}}
					\addlegendentry{target \texttt{tol}}
					\addlegendentry{Averaged}
					\addlegendentry{Averaged+Newton}
					\addlegendentry{Anderson}
					\addlegendentry{Anderson+Newton}
					\addlegendentry{T-map}				
					%				\addlegendentry{Gradient}
				\end{semilogyaxis}
			\end{tikzpicture}	
			\caption{Backward error  $\max\, \{|x_i^{(k)}-[T(x^{(k)})]_i| / |x_i^{(k)}|$ for $1\leq i \leq n\}$}
			\label{fig:exp_atc_bwd}
		\end{subfigure}%
		%	\hfill
		
		\begin{subfigure}{1\textwidth}%
			\begin{tikzpicture}
				\begin{semilogyaxis}[
					xlabel={iterations $k$},
					ylabel={forward error},
					width=13cm,height=6.5cm,
					ymin=10^-16,ymax=2000,
					xmin=0,xmax=70,
					xlabel style={yshift=2mm, },
					xticklabel style={font=\small,},
					yticklabel style={font=\small,/pgf/number format/fixed,},
					legend pos = south west,
					legend cell align={left},
					legend style={font=\tiny, nodes={inner sep=1pt}},
					]
					
					\addplot [color=black, mark = none, dotted, line width = 1.5pt] table {\pathtodir/err_bounds.txt};
					\addplot [color=green!70, mark = square, mark size = 1.2pt, mark options={solid}, densely dashed] table {\pathtodir/err_sol_2avgNN.txt};
					\addplot [color=green!70, mark = square*, mark size = 1.2pt] table {\pathtodir/err_sol_2avg.txt};
					\addplot [color=red, densely dashed, mark = o, mark size = 1.2pt,  mark options={solid}, densely dashed] table {\pathtodir/err_sol_3AndNN.txt}; 
					\addplot [color=red, mark = *, mark size = 1.2pt] table {\pathtodir/err_sol_3And.txt};
					\addplot [color=blue, mark = x, mark size = 1.9pt] table {\pathtodir/err_sol_1Tmap.txt}; 
					%				\addplot [color=purpleG, mark = triangle] table {\pathtodir/err_sol_4grad.txt}; 
					
					\addlegendentry{upper bound}
					\addlegendentry{Averaged}
					\addlegendentry{Averaged+Newton}
					\addlegendentry{Anderson}
					\addlegendentry{Anderson+Newton}
					\addlegendentry{T-map}
					%				\addlegendentry{Gradient}			
				\end{semilogyaxis}
			\end{tikzpicture}	
			\caption{Forward error $\max\, \{|x_i^{(k)}-y_i| / |y_i|$ for $1\leq i \leq n\}$}
			\label{fig:exp_atc_fwd}
		\end{subfigure}%
		\caption{Error comparison of six different strategies for computing Fitness Centrality on the Air Traffic Control network (directed)}
		\label{fig:exp_atc}
	\end{figure}

	\paragraph{BACI HS92--Y1996 trade network}
	In our last experiment, we study the trade network that has been the motivating example for developing Economic Fitness Complexity \cite{def2EFC2013}. We consider the BACI dataset\footnote{Available at \url{https://www.cepii.fr/CEPII/en/bdd_modele/bdd_modele_item.asp?id=37}, January 2026 version.} HS92 for the year 1996 \cite{BACI}. Following the procedure described in \cite[Methods]{def1EFC2012}, we consider $n_c=213$ countries and $n_p=1241$ products and construct the country-product matrix $M$ of size $n_c\times n_p$. The underlying idea is that nodes $c$ and $p$ are connected, i.e., $M_{cp}=1$, if the country $c$ is a significant exporter of the product~$p$. Then, we create an undirected, unweighted, and bipartite graph on $n_c+n_p = 1454$ vertices, with adjacency matrix 
	\[ A= \begin{pmatrix}0 & M \\ M^\top & 0\end{pmatrix}.\]
	
	The plots of the errors for the different strategies are shown in Figure~\ref{fig:exp_BACI96}. 
	The results of this experiment are qualitatively different compared to the previous ones. Indeed, all stationary methods exhibit a high convergence rate. The only method which stagnates is the gradient descent one, being suboptimal in this case. Among the stationary methods, the T-map strategy is the fastest one, followed closely by the Anderson strategy. From Figure~\ref{fig:exp_BACI96_fwd}, it is clear that the upper bound remains practically constant, overshooting the actual forward errors by far. 
	Switching to Newton steps further reduces the total iteration count, as can be seen with the averaged+Newton and Anderson+Newton strategies. 
	%%% BACI 96, undirected, bipartite
	\begin{figure}
		\def\pathtodir{dataset/BACI96}
		\begin{subfigure}{1\textwidth}
			\begin{tikzpicture}
				\begin{semilogyaxis}[
					xlabel={iterations $k$},
					ylabel={backward error},
					width=13cm, height=6.5cm,
					ymin=10^-16, ymax=100,
					xmin=0,xmax=70,
					xlabel style={yshift=2mm, },
					xticklabel style={font=\small,},
					yticklabel style={font=\small,/pgf/number format/fixed,},
					legend pos = south west,
					legend cell align={left},
					legend style={font=\tiny, nodes={inner sep=1pt}},
					]
					
					\addplot[mark=none, dashed, thin, black, samples=2, domain=0:100] {1e-2}; \label{tolN}% horizontal line			
					\addplot[mark=none, dotted, thin, black, samples=2, domain=0:100] {1e-10}; \label{tol}% horizontal line			
										
					\addplot [color=blue, mark = x, mark size = 1.9pt] table {\pathtodir/err_iter_1Tmap.txt};				
					\addplot [color=green!70, mark = square, mark size = 1.2pt, mark options={solid}, densely dashed] table {\pathtodir/err_iter_2avgNN.txt};
					\addplot [color=green!70, mark = square*, mark size = 1.2pt] table {\pathtodir/err_iter_2avg.txt};
					\addplot [color=red, densely dashed, mark = o, mark size = 1.2pt,  mark options={solid}, densely dashed] table {\pathtodir/err_iter_3AndNN.txt};
					\addplot [color=red, mark = *, mark size = 1.2pt] table {\pathtodir/err_iter_3And.txt};
					\addplot [color=purpleG, mark = triangle] table {\pathtodir/err_iter_4grad.txt};
					
					\addlegendentry{\texttt{tol\_Newton}}
					\addlegendentry{target \texttt{tol}}
					\addlegendentry{T map}
					\addlegendentry{Averaged}
					\addlegendentry{Averaged+Newton}
					\addlegendentry{Anderson}
					\addlegendentry{Anderson+Newton}
					\addlegendentry{Gradient}
				\end{semilogyaxis}
			\end{tikzpicture}	
			\caption{Backward error  $\max\, \{|x_i^{(k)}-[T(x^{(k)})]_i| / |x_i^{(k)}|$ for $1\leq i \leq n\}$}
			\label{fig:exp_BACI96_bwd}
		\end{subfigure}%
		%	\hfill
		
		\begin{subfigure}{1\textwidth}%
			\begin{tikzpicture}
				\begin{semilogyaxis}[
					xlabel={iterations $k$},
					ylabel={forward error},
					width=13cm,height=6.5cm,
					ymin=10^-16,ymax=10^4,
					xmin=0,xmax=70,
					xlabel style={yshift=2mm, },
					xticklabel style={font=\small,},
					yticklabel style={font=\small,/pgf/number format/fixed,},
					legend pos = south west,
					legend cell align={left},
					legend style={font=\tiny, nodes={inner sep=1pt}},
					]
					
					\addplot [color=black, mark = none, dotted, line width = 1.5pt] table {\pathtodir/err_bounds.txt};
					\addplot [color=blue, mark = x, mark size = 1.9pt] table {\pathtodir/err_sol_1Tmap.txt}; 
					\addplot [color=green!70, mark = square, mark size = 1.2pt, mark options={solid}, densely dashed] table {\pathtodir/err_sol_2avgNN.txt};
					\addplot [color=green!70, mark = square*, mark size = 1.2pt] table {\pathtodir/err_sol_2avg.txt};
					\addplot [color=red, densely dashed, mark = o, mark size = 1.2pt,  mark options={solid}, densely dashed] table {\pathtodir/err_sol_3AndNN.txt}; 
					\addplot [color=red, mark = *, mark size = 1.2pt] table {\pathtodir/err_sol_3And.txt};
					\addplot [color=purpleG, mark = triangle] table {\pathtodir/err_sol_4grad.txt}; 
					
					\addlegendentry{upper bound}
					\addlegendentry{T map}
					\addlegendentry{Averaged}
					\addlegendentry{Averaged+Newton}
					\addlegendentry{Anderson}
					\addlegendentry{Anderson+Newton}
					\addlegendentry{Gradient}			
				\end{semilogyaxis}
			\end{tikzpicture}	
			\caption{Forward error $\max\, \{|x_i^{(k)}-y_i| / |y_i|$ for $1\leq i \leq n\}$}
			\label{fig:exp_BACI96_fwd}
		\end{subfigure}%
		\caption{Error comparison of six different strategies for computing Fitness Centrality on the BACI HS92-Y1996 trade network (undirected, bipartite)}
		\label{fig:exp_BACI96}
	\end{figure}
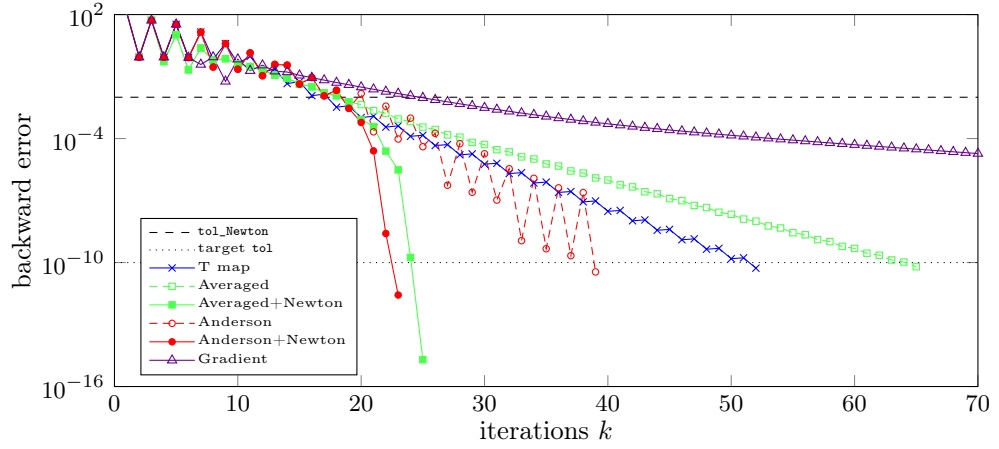
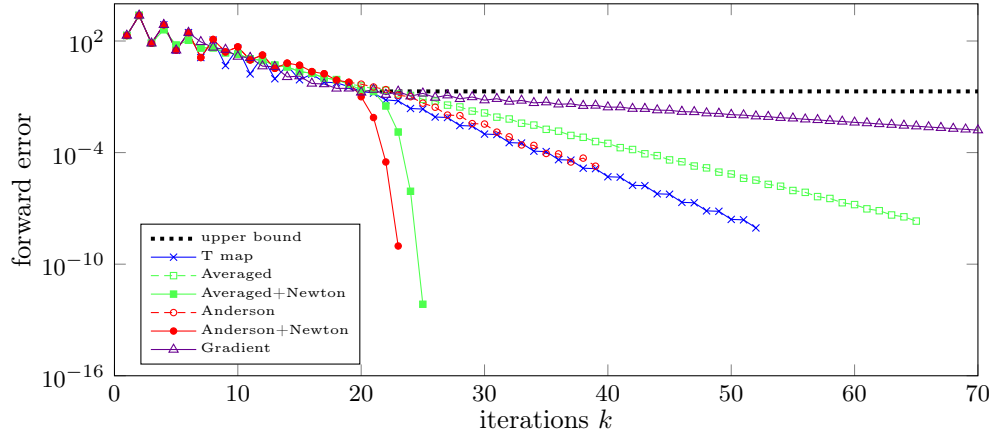
	
	\paragraph{Final remarks on experiments and computational aspects}
	In all the experiments, we have seen that using Newton's method significantly aids convergence. Therefore, we recommend using it for the final iteration steps, regardless of the previously used computational strategy.
	
	The Anderson strategy is usually faster than the averaged strategy, but it requires a good choice of parameters $q,m$. It might be necessary to try different combinations before achieving a satisfactory result. We recommend using Anderson as the initial method if convergence speed is of utmost importance; otherwise, the user may prefer the slower, but more stable, averaged strategy. 

 Concerning computational costs, both Anderson and especially Newton acceleration incur a higher per-iteration cost compared to the other methods. The overhead depends on the size, sparsity, and structural properties of the Jacobian for Newton's method, and on the choice of parameters for Anderson. Nevertheless, numerical experiments on larger networks, both synthetic and from actual applications, clearly show the advantage of these acceleration techniques in terms of both CPU time and accuracy for sufficiently large networks (typically, with a few thousand nodes).

	\section{Conclusions and open questions}\label{sec:conclusions}
	
	 In this paper we have established the existence and uniqueness 
	of the Fitness Centrality vector for general graphs and given a rigorous global convergence proof of the (regularized) Fitness Centrality iteration algorithm.
	Our analysis results in a bound on the error of the iterates that depends on the regularization parameter, showing that smaller values of this parameter can
	be expected to result in slow convergence. 
	Furthermore, we have proposed several strategies with improved convergence behavior, together with acceleration strategies of Newton and Anderson type. These
	techniques significantly improve the robustness and overall efficiency of Fitness Centrality.  In the case of undirected graphs we have included comparisons with
	a gradient descent technique, which we found to be generally less efficient than the accelerated methods for general graphs.
	
	We conclude by mentioning a few questions that could be addressed in further work.
	
	\begin{itemize}
		\item We have seen in Corollary~\ref{cor:limit of x^k is in E*} that the Fitness Centrality vector $y$ belongs to the domain $E_*$, which is, unfortunately, a large region. Is it possible to find more accurate estimates for the location of $y$, possibly depending on the nature of the graph? 
		\item In the introduction of the Non-Homogeneous Economic Fitness Complexity \cite{regularizedEFC2018}, the authors propose to use parameters $\phi_c$ and $\pi_p$ for each country-node and product-node respectively; they interpret $\phi_c$ as the intrinsic fitness of a country, and $\pi_p$ as an innovation threshold. Later, they set all parameters equal to $\delta$, leading to \eqref{eq:FC iteration}, with the goal of recovering the homogeneous Economic Fitness Centrality in the limit of small $\delta$. 
		
		A similar construction can be carried out for the Fitness Centrality of a generic graph. Consider a personalization vector $v$ and the corresponding iteration map 
		\[T_v(x) = v + A \cdot \inv(x).\]
		After iterating this map and reaching the fixed point of $T_v$, one may ask: Is there an interpretation for the obtained Fitness Centrality with a personalization vector~$v$? Can a particular choice of $v$ give insight into the nodes' centrality measures?
		
		\item Although stationary methods converge very slowly on most graphs we tested, they achieve rapid convergence on certain graphs, 
		such as the BACI trade network of Figure~\ref{fig:exp_BACI96}. 
		It would be interesting to characterize families of graphs for which rapid convergence of the basic iteration is attained. The rank of the adjacency matrix
		 of the graph might play a role here, with convergence being faster for matrices of relatively low rank.
	\end{itemize}
	
	\backmatter
	
	\section*{Statements and Declarations}	

	\bmhead{Acknowledgements}
	
	The authors are members of the Gruppo Nazionale di Calcolo Scientifico (GNCS) of the Istituto Nazionale di Alta
	Matematica (INdAM).  
	
	\bmhead{Competing Interests}
	The authors declare that they have no competing interests. 
	
	\bmhead{AI disclosure}
	The authors used Google Gemini for grammar checking and copy-editing purposes. The mathematical results have been obtained without the use of AI or LLMs. 
	
	\bmhead{Funding}
	This work supported in part by ``Fondi per la ricerca di base" of Scuola Normale Superiore.

		\bibliography{bibliography}
		
		\end{document}